\documentclass{amsart}

\usepackage{amsmath,amscd,amssymb}

\usepackage{hyperref}
\hypersetup{colorlinks,linkcolor={red},citecolor={blue},urlcolor={blue}}

\newtheorem{theorem}{Theorem}[section]

\newtheorem{proposition}[theorem]{Proposition}
\newtheorem{corollary}[theorem]{Corollary}

\theoremstyle{definition}
\newtheorem{definition}[theorem]{Definition}

\newtheorem{remark}[theorem]{Remark}

\numberwithin{equation}{section}

\newcommand{\CC}{\mathbb C}
\newcommand{\HH}{\mathbb H}

\newcommand{\NN}{\mathbb N}
\newcommand{\PP}{\mathbb P}
\newcommand{\QQ}{\mathbb Q}
\newcommand{\RR}{\mathbb R}
\newcommand{\ZZ}{\mathbb Z}

\newcommand{\cD}{\mathcal D}
\newcommand{\cH}{\mathcal H}

\newcommand{\SL}{\mathop{\mathrm {SL}}\nolimits}
\newcommand{\SO}{\mathop{\mathrm {SO}}\nolimits}
\newcommand{\Sp}{\mathop{\mathrm {Sp}}\nolimits}
\newcommand{\Orth}{\mathop{\null\mathrm {O}}\nolimits}

\newcommand{\rank}{\mathop{\mathrm {rank}}\nolimits}

\newcommand{\latt}[1]{{\langle{#1}\rangle}}

\newcommand{\Borch}{\operatorname{Borch}}
\def\div{\operatorname{Div}}
\newcommand{\m}{\operatorname{mod}}
\newcommand{\im}{\operatorname{Im}}

\newcommand{\mult}{\operatorname{mult}}

\newcommand{\sing}{\operatorname{Sing}}
\newcommand{\II}{\operatorname{II}}

\newcommand{\inv}{\operatorname{inv}}
\newcommand{\Sch}{\operatorname{Sch}}

\begin{document}
\title[Antisymmetric paramodular forms of weight 3]{Antisymmetric paramodular forms of weight 3}

\author{Valery Gritsenko}

\address{Laboratoire Paul Painlev\'{e}, Universit\'{e} de Lille and IUF, 59655 Villeneuve d'Ascq Cedex, France and National Research University Higher School of Economics, Russian Federation}

\email{valery.gritsenko@univ-lille.fr}

\author{Haowu Wang}

\address{Laboratoire Paul Painlev\'{e}, Universit\'{e} de Lille, 59655 Villeneuve d'Ascq Cedex, France}

\email{haowu.wangmath@gmail.com}


\date{\today}


\begin{abstract}
The problem on the construction of antisymmetric paramodular forms of 
canonical weight $3$ was open since 1998. Any cusp form of this type 
determines a canonical differential form on any smooth compactification of 
the moduli space of Kummer surfaces associated to $(1,t)$-polarised 
abelian surfaces.
In this paper, we construct the first infinite family of antisymmetric 
paramodular forms of weight $3$  as Borcherds products whose first 
Fourier--Jacobi coefficient is a theta block. 
\end{abstract}

\maketitle

\section{Introduction}
Let $t$ be a positive integer. The paramodular group of level (or polarisation) $t$
is the integral symplectic group of the skew-symmetric form with elementary divisors $(1, t)$.
This group is conjugate to  a subgroup $\Gamma_t$  of the rational symplectic group $\Sp_2(\QQ)$ (see \S \ref{sec:Siegel}). The Siegel modular threefold 
$
{\mathcal A}_t=\Gamma_t\setminus \HH_2
$,
where $\HH_2$ is the Siegel upper half space of genus $2$, 
is isomorphic to the moduli space of abelian surfaces with a polarisation of type $(1,t)$. This moduli space is not compact.
If $F$ is a cusp form of weight $3$ with respect to $\Gamma_t$, 
then $\omega_F=F(Z)dZ$ is a holomorphic $3$-form on ${\mathcal A}_t$.
According to Freitag's criterion (see \S \ref{sec:6}), $\omega_F$ can be extended to any smooth compactification $\overline{{\mathcal A}_t}$ of the moduli space. Therefore  
$$
h^{3,0}({\overline{{\mathcal A}_t}})=\dim_\CC (S_3(\Gamma_t)),
$$
where $S_3(\Gamma_t)$ is the space of paramodular cusp forms of canonical weight $3$.
The lifting construction proposed by Gritsenko in \cite{G94}
provides cusp forms of weight $3$ with respect to the paramodular group
$\Gamma_t$ for all $t$ except $20$ polarisations
$$
t=1,\ldots,12, \ 14,\ 15,\ 16,\ 18,\  20,\ 24, \ 30, \ 36. 
$$
In particular, $H^3(\Gamma_t,\CC)$ is not trivial  for all non exceptional 
polarisations. We note  that
$\dim S_3\left(\Gamma_t\right) =0$ for these twenty~$t$ (see \cite{BPY}). 
Due to the existence of canonical differential forms, the moduli space of 
$(1,t)$-polarised abelian surfaces might have trivial geometric genus only 
for the twenty exceptional polarisations. For $t\leq 20$ 
the rationality or  unirationality  of the moduli space is known 
(see \cite{GP}).

The paramodular group $\Gamma_t$ is not a maximal discrete group acting on $\HH_2$ if $t\neq 1$. It has a normal extension $\Gamma_t^*$ such that
$\Gamma_t^*/\Gamma_t\cong (\ZZ/2\ZZ)^{\nu(t)}$,
where $\nu(t)$ is the number of distinct prime divisors of $t$ (see \cite{GH98}).
In \cite[Theorem 1.5]{GH98} it was proved that the modular variety   
${\mathcal K}_t=\Gamma_t^*\setminus \HH_2$  
can be considered as the moduli space of 
Kummer surfaces associated to $(1,t)$-polarised abelian surfaces.
We note that the birational geometry of  moduli spaces  of Kummer surfaces is much more complicated than the geometry of moduli spaces 
of polarised abelian surfaces because the ramification divisor of the modular variety $\Gamma_t^*\setminus \HH_2$ is much larger (see \cite{GHS2}).
We expect a long list of the moduli spaces ${\mathcal K}_t$ of Kummer surfaces which are rational or unirational since the first cusp $\Gamma_t^*$-form of weight $3$ is known only for $t=167$ (see \cite{GPY2}). 
We note that  the uniruledness of  ${\mathcal K}_t$
for a non-exceptional $t=21$ was 
proved in \cite{GH14}. 

If $t=p$ is a prime, then 
$
\Gamma_t^*=\Gamma_t^+=\Gamma_t\cup \Gamma_tV_t
$
contains only one additional involution $V_t$. 
A $\Gamma_t$-modular form $F$ of  weight $3$ will be modular with respect to the double extension $\Gamma_t^+$ if it  satisfies 
an additional  functional equation (see \S \ref{sec:Siegel} for more details)
\begin{equation}\label{anti1}
F(\left(
\begin{smallmatrix} t\omega &z\\z&\tau/t\end{smallmatrix}
\right))=-F (\left(
\begin{smallmatrix} \tau &z\\z&\omega\end{smallmatrix}
\right)).
\end{equation}
We call such $\Gamma_t$-paramodular forms {\bf antisymmetric}.
We note that the modular form obtained by Gritsenko's lifting 
are symmetric, i.e. they satisfy the equation of type 
\eqref{anti1} with sign plus.

The problem of the construction of antisymmetric paramodular 
forms of weight $3$ was open since 1998. 
For the Siegel modular group $\Gamma_1=\Sp_2(\ZZ)$, there is essentially 
only one 
antisymmetric modular form. This is the Igusa modular form
$\Delta_{35}$  
of odd weight $35$. The Borcherds product expansion for 
$\Delta_{35}$ was proposed in \cite{GN96}. 

The theory of automorphic 
products gives a powerful instrument to construct antisymmetric cusp forms.  
The first six examples of weight $3$ for 
$t=122$, $167$, $173$, $197$, $213$, $285$
were constructed in \cite{GPY2} as automorphic Borcherds products.
This sporadic construction was originally proposed for 
weight $2$ as an answer on a question related to
the Brumer--Kramer conjecture on modularity of abelian surfaces
(see \cite{BK}). 

In this paper we find the first {\bf infinite} series of antisymmetric
paramodular forms of weight $3$ (see Theorem \ref{th:wt3} in \S \ref{sec:Siegel} and \S \ref{sec:wt3}). 
The series  starts with a non-cusp form
for $t=98$. Its first cusp form for $t=122$ coincides with the example 
constructed in \cite{GPY2}. 
As an application (see \S \ref{sec:6}) we prove that   
$H^{3,0}(\Gamma_t^+\setminus\HH_2, \CC)$  or  $H^3(\Gamma_t^+, \CC)$ 
is nontrivial for all square-free $t$ from the  infinite series 
presented in Theorem \ref{th:wt3}. 

The infinite series of antisymmetric paramodular forms is related to a 
very special reflective modular form in $8$ variables on an indefinite orthogonal group $\Orth(2,8)$.
This modular form  $\Phi_3^{\Sch}$ is an automorphic Borcherds product 
(see \S \ref{sec:lift} and \S \ref{sec:wt3}). It 
was discovered  by Nils Scheithauer 
in  \cite[Section 10]{Sch06} 
in the framework of his fundamental program on 
the classification of reflective modular forms of singular weight
(see \cite{Sch06}--\cite{Sch17}). 
The function $\Phi_3^{\Sch}$ is  similar to the  
Borcherds form $\Phi_{12}$ on $\Orth(2,26)$ 
which determines the Fake--Monster 
Lie algebra and plays a crucial role in the Borcherds proof of the 
moonshine conjecture (see \cite{Bo95}--\cite{Bo98}).

The original Scheithauer's construction was given at a zero-dimensional 
cusp of the corresponding  modular variety of orthogonal type
as the Borcherds product of a certain nearly holomorphic modular form with respect 
to the Hecke congruence subgroup $\Gamma_0(7)$. In \S \ref{sec:wt3} we find 
another construction of the Scheithauer modular form at a
one-dimensional cusp in a way proposed in \cite{GN17} and \cite{G18}. 
It turns out that the first Fourier--Jacobi coefficient of the Borcherds product at this cusp 
is a holomorphic Jacobi form which coincides with 
the Kac--Weyl denominator function of the affine Lie algebra 
$\hat{\mathfrak g}(A_6)$. As a corollary we get
that the corresponding Lorentzian Kac--Moody algebra is a hyperbolization
of the affine Lie algebra $\hat{\mathfrak g}(A_6)$
(see \S \ref{sec:6}).
In the last \S \ref{sec:wt4} we consider one more example of this type 
related to the root system $A_4\oplus A_4$ and construct  an infinite 
family of antisymmetric paramodular forms of weight $4$.

\section{Theta blocks and the main theorem}\label{sec:Siegel}
First we recall the definition of Siegel paramodular forms.
Let  
$$
\HH_2=\{Z=\left( \begin{array}{cc}
  \tau & z \\ 
  z & \omega
  \end{array}  \right)\in M(2,\CC): \im Z>0\}
$$
be the Siegel upper half space of genus $2$. The real symplectic group 
$\Sp_2(\RR)$ acts on $\HH_2$ via 
$$ M\latt{Z}=(AZ+B)(CZ+D)^{-1}, \quad M=\left(\begin{array}{cc}
A & B \\ 
C & D
\end{array}  \right) \in \Sp_2(\RR). $$
Let $k\in \ZZ$. We define the slash operator 
on the space of holomorphic functions on $\HH_2$ in the usual way
\begin{equation}
(F\lvert_k M )(Z) =\det(CZ+D)^{-k}F(M\latt{Z}).
\end{equation}
Let $t$ be a positive integer. The \textit{paramodular group} of level $t$ 
is a subgroup of $\Sp_2(\QQ)$ defined as
\begin{equation}
\Gamma_t=\left(\begin{array}{cccc}
  * & t* & * & * \\ 
  * & * & * & */t \\ 
  * & t* & * & * \\ 
  t* & t* & t* & *
  \end{array}   \right) \cap \Sp_2(\QQ),\quad \text{all }\  *\in \ZZ.
\end{equation}

This group is conjugate to the integral symplectic group of the skew-symmetric form with elementary divisors $(1, t)$ (see \cite{GH98, GNII}).
As we mentioned in the introduction, the quotient
$
{\mathcal A}_t=\Gamma_t\setminus \HH_2
$
is isomorphic to the moduli space of abelian surfaces with a polarisation 
of type $(1,t)$.

For $t>1$, we shall use the following double normal extension of $\Gamma_t
$ in $\Sp_2(\RR)$
\begin{equation}\label{Gamma+}
\Gamma_t^+=\Gamma_t\cup \Gamma_t V_t,\qquad V_t=\frac{1}{\sqrt{t}}
\left(\begin{array}{cccc}
0 & t & 0 & 0 \\ 
-1 & 0 & 0 & 0 \\ 
0 & 0 & 0 & 1 \\ 
0 & 0 & -t & 0
\end{array}  \right).
\end{equation}

\noindent
{\bf Definition.}
A holomorphic function $F: \HH_2 \longrightarrow \CC$ is called a 
\textit{Siegel paramodular form} of  weight $k$ and level $t$ if $F
\lvert_k M=F$ for any $M\in \Gamma_t$.
We denote the space of such modular forms by $M_k(\Gamma_t)$. A 
paramodular form $F$ is called a \textit{cusp} form if 
$\Phi(F\lvert_k g)=0$ for all $g\in \Sp_2(\QQ)$, here $\Phi$ is the Siegel 
operator. The space of paramodular cusp forms is denoted by 
$S_k(\Gamma_t)$.
\smallskip

Let $\chi_t: \Gamma_t^+\to \{\pm 1\}$ be the nontrivial character with kernel $\Gamma_t$. By virtue of this character, $M_k(\Gamma_t)$ is decomposed into the direct sum of plus and minus $V_t$-eigenspaces, i.e.
$
M_k(\Gamma_t)=M_k(\Gamma_t^+)\oplus M_k(\Gamma_t^+, \chi_t).
$

For $F\in M_{k}(\Gamma_t^+, \chi_t^\epsilon)$ with $\epsilon=0$
or $1$, we consider its Fourier and Fourier--Jacobi expansions 
\begin{equation}\label{FJ}
F(Z)=\sum_{m\geq 0} \sum_{\substack{n\in \NN, r\in \ZZ\\ 4nmt-r^2\geq 0}}c(n,r,m)q^n\zeta^r \xi^{mt}
=\sum_{m\geq 0}\phi_{mt}(\tau,z)\xi^{mt},
\end{equation}
where $q=\exp(2\pi i \tau)$, $\zeta=\exp(2\pi i z)$, $\xi=\exp(2\pi i \omega)$.
One can prove (see \cite{G94}) that $F$ is cusp form if 
$c(n,r,m)\ne 0$ implies that $4nmt-r^2>0$.

Then we see that each Fourier--Jacobi coefficient is a holomorphic Jacobi form of weight $k$ and index $mt$ in the sense of Eichler--Zagier \cite{EZ}, 
namely $\phi_{mt}\in J_{k,mt}$ (see \S \ref{sec:Jacobi} for more details). 
Moreover, according to the  action of the involution $V_t$, we get the equality 
$$
(-1)^{k+\epsilon}F(\tau,z,\omega)=F(\omega t,z,\tau/t),
$$ 
which yields $c(n,r,m)=(-1)^{k+\epsilon}c(m,r,n)$ (compare with \eqref{anti1}).
When $k+\epsilon$ is even, $F$ is called \textit{symmetric}. 
When $k+\epsilon$ is odd, $F$ is called \textit{antisymmetric}.

The paramodular forms constructed by additive Jacobi lifting due to 
Gritsenko \cite{G94} are always symmetric. Thus the only regular way to 
construct antisymmetric paramodular forms is the method called Borcherds 
automorphic product (see \cite{Bo98}). In the Gritsenko--Nikulin interpretation of 
Borcherds product given in \cite{GNII} one can control the action of the 
involution $V_t$ in terms of  the Fourier coefficients of 
weakly holomorphic Jacobi forms of weight $0$. Unfortunately, one cannot 
produce any infinite series of such weakly holomorphic Jacobi forms because
usually one gets meromorphic automorphic products.
An attempt to overcome this difficulty was made in article \cite{GPY2}, 
using the theory of theta blocks (see \cite{GSZ}). This sporadic method 
gives natural candidates for the first Fourier--Jacobi coefficient of 
an antisymmetric paramodular form of weight $2$ or  $3$. As a result
it was constructed an infinite series of antisymmetric paramodular forms
with weights going to infinity. The first members of the constructed 
series  (see Table 1 in \cite{GPY2}) are of weight $2$ (three examples
for $t=587$, $713$ and $893$) and weight $3$ (six examples 
$t=122$, $167$, $173$, $197$, $213$, $285$ mentioned in the introduction).

In this paper  we construct antisymmetric paramodular forms using
pull-backs of two special antisymmetric orthogonal modular forms of higher 
dimension. Like in \cite{GW, GW18} we use the construction of holomorphic theta blocks in many variables.  

Let 
$$
\eta(\tau)=q^{1/24}\prod_{n\geq 1}(1-q^n)\in S_{1/2}(\SL_2(\ZZ), v_\eta)
$$ 
be the Dedekind $\eta$-function. This is a cusp form of weight $1/2$
with the multiplier system $v_\eta: \SL_2(\ZZ)\to U_{24}$ of order $24$. 
We consider the odd Jacobi theta-series 
\begin{equation}\label{theta}
\vartheta(\tau,z)= q^{\frac{1}{8}}(\zeta^{\frac{1}{2}}-\zeta^{-\frac{1}
{2}}) \prod_{n\geq 1} (1-q^{n}\zeta)(1-q^n \zeta^{-1})(1-q^n).
\end{equation}
It is known that 
$\vartheta(\tau,-z)=-\vartheta(\tau,z)$ and  
$\vartheta(\tau,z)\in J_{1/2, 1/2}(v_\eta^3\times v_H)$ is a holomorphic Jacobi form 
of weight $1/2$ and index $1/2$ (see \cite{GNII}).
We define a theta block
\begin{equation}
\Theta_f=\eta^{f(0)} \prod_{a=1}^{\infty} 
\left( \vartheta_a/ \eta \right)^{f(a)},
\end{equation}
where $f: \NN \to \NN$ is a sequence with finite support and 
$\vartheta_{a}=\vartheta(\tau, az)$.

The quotient $\Theta_f$ is a weak Jacobi form of weight $f(0)/2$ with a 
character or multiplier system. For some function $f$
it is a {\it holomorphic} Jacobi form. 
The simplest example is the theta-quark (see \cite{CG} and \cite{GSZ})
$$
\vartheta_a\vartheta_b\vartheta_{a+b}/\eta\in J_{1, a^2+ab+b^2}(v_\eta^8).
$$

In this paper we prove the following theorem.

\begin{theorem}\label{th:wt3}
For $\mathbf{a}=(a_1,a_2,a_3,a_4,a_5,a_6)\in \ZZ^6$,
the theta block
\begin{equation}\label{TB3}
\begin{split}
\Theta_\mathbf{a}=&\vartheta_{a_1}\vartheta_{a_2}\vartheta_{a_3}\vartheta_{a_4}
\vartheta_{a_5}\vartheta_{a_6}\vartheta_{a_1+a_2}\vartheta_{a_2+a_3}
\vartheta_{a_3+a_4}\vartheta_{a_4+a_5}\vartheta_{a_5+a_6}
\vartheta_{a_1+a_2+a_3}\\
&\vartheta_{a_2+a_3+a_4}\vartheta_{a_3+a_4+a_5}
\vartheta_{a_4+a_5+a_6}\vartheta_{a_1+a_2+a_3+a_4}
\vartheta_{a_2+a_3+a_4+a_5}\\
&
\vartheta_{a_3+a_4+a_5+a_6}\vartheta_{a_1+a_2+a_3+a_4+a_5}
\vartheta_{a_2+a_3+a_4+a_5+a_6}\\
&\vartheta_{a_1+a_2+a_3+a_4+a_5+a_6} /\eta^{15}=q^2(\cdots)  \in  J_{3, N(\mathbf{a})}
\end{split}
\end{equation}
of type $\frac{21-\vartheta}{15-\eta}$ is a holomorphic Jacobi form of weight $3$ and index $N(\mathbf{a})$,
where 
\begin{equation}\label{N(a)}
\begin{split}
N(\mathbf{a})
=&3a_1^2 + 5a_2a_1 + 4a_3a_1 + 3a_4a_1 + 2a_5a_1+a_6a_1+5a_2^2\\
&+ 8a_3a_2+ 6a_4a_2 + 4a_5a_2 + 2a_6a_2 + 6a_3^2 + 9a_4a_3 + 6a_5a_3\\
&+ 3a_6a_3+ 6a_4^2 + 8a_5a_4 + 4a_6a_4 +5a_5^2 + 5a_6a_5 + 3a_6^2.
\end{split}
\end{equation}

If this theta block is not identically zero, there 
exists an antisymmetric holomorphic paramodular form 
$F_{\mathbf{a}}\in M_3(\Gamma_{N(\mathbf{a})}^+)$ of weight $3$ and 
level $N(\mathbf{a})$ whose leading Fourier--Jacobi coefficient is the 
above theta block. Moreover, $F_{\mathbf{a}}$ is a cusp form if 
$N(\mathbf{a})$ is square-free.
\end{theorem}

\section{Jacobi forms of lattice index and Borcherds products}
\label{sec:Jacobi}
In this section, we introduce modular forms on orthogonal groups and Jacobi forms in many variables which will be used in the proof of Theorem \ref{th:wt3} (see \cite{CG} or \cite{G18} for more details). 

We consider an even integral lattice
$M=U\oplus U_1\oplus L(-1)$  of signature $(2,n)$ with $n\geq 3$,
where $U$, $U_1$ are two hyperbolic planes and $L$ is an even positive 
definite integral lattice. We fix a basis of $M$ of the form $(e,e_1,...,f_1,f)$, 
where $U=\ZZ e+\ZZ f$, $U_1=\ZZ e_1+\ZZ f_1$, and $...$ denotes a basis of 
$L(-1)$. Let
\begin{equation*}
\cD(M)=\{[\omega] \in  \PP(M\otimes \CC):  (\omega, \omega)=0, (\omega,
\bar{\omega}) > 0\}^{+}
\end{equation*}
be the associated Hermitian symmetric domain of type IV (here $+$ denotes 
one of its two connected components). Let us denote the index $2$ subgroup 
of the orthogonal group $\Orth(M)$ preserving $\cD(M)$ by $\Orth^+ (M)$.    

Let $\Gamma$ be a finite index subgroup of $\Orth^+ (M)$ and $k\in \ZZ$. A modular form of weight $k$ and character $\chi: \Gamma\to \CC^*$ with respect to $\Gamma$ is a holomorphic function $F: \cD(M)^{\bullet}\to \CC$ on the affine cone $\cD(M)^{\bullet}$ satisfying
\begin{align*}
F(t\mathcal{Z})&=t^{-k}F(\mathcal{Z}), \quad \forall t \in \CC^*,\\
F(g\mathcal{Z})&=\chi(g)F(\mathcal{Z}), \quad \forall g\in \Gamma.
\end{align*}
A modular form is called a cusp form if it vanishes at every cusp (i.e. a boundary component of the Baily-Borel compactification of the modular variety $\Gamma\backslash \cD(M)$). 

Let $D(M)=M^\vee / M$ be the discriminant group of $M$. We denote the 
stable orthogonal group which is the subgroup of $\Orth^+ (M)$ acting 
trivially on $D(M)$ by $\widetilde{\Orth}^+(M)$. For any $v\in M\otimes \QQ
$ satisfying $(v,v)<0$, we define the rational quadratic divisor 
associated to $v$ as
\begin{equation*}
 \cD_v=\{ [Z]\in \cD(M) : (Z,v)=0\}. 
\end{equation*}
A reflective modular form is a modular form on $\cD(M)$ whose zero divisor is a union of rational quadratic divisors associated to primitive vectors determining reflections in $\Orth^+(M)$ (see e.g. \cite{GN17} or \cite{G18} for the exact definition).

We fix a tube realization of the homogenous domain $\cD(M)$ related to 
the 1-dimensional boundary component defined by the isotropic subspace $P=
\latt{e,e_1}$
$$
\cH(L)=\{Z=(\tau,\mathfrak{z},\omega)\in \HH\times (L\otimes\CC)\times 
\HH: (\im Z,\im Z)>0\},
$$
where $(\im Z,\im Z)=2\im \tau \im \omega - (\im \mathfrak{z},\im 
\mathfrak{z})$. In this setting,  a Jacobi form can be viewed as a modular 
form with respect to the Jacobi group $\Gamma^J(L)$ which is a 
distinguished parabolic subgroup $\{ g\in \SO(M)^+: gP=P, g\lvert_L=id\} < 
\Orth^+ (M)$ (see \cite{CG}). The Jacobi group is the semidirect product 
of $\SL_2(\ZZ)$ with the Heisenberg group $H(L)$ of $L$.

\begin{definition}
Let $\varphi : \HH \times (L \otimes \CC) \rightarrow \CC$ be a 
holomorphic function and $k\in\ZZ$, $t\in \NN$. If $\varphi$ satisfies the 
functional equations
\begin{align*}
\varphi \left( \frac{a\tau +b}{c\tau + d},\frac{\mathfrak{z}}{c\tau + d} 
\right)& = (c\tau + d)^k e^{i \pi t \frac{c(\mathfrak{z},\mathfrak{z})}{c 
\tau + d}} \varphi ( \tau, \mathfrak{z} ), \quad \left( \begin{array}{cc}
a & b \\ 
c & d
\end{array} \right)   \in \SL_2(\ZZ),\\
\varphi (\tau, \mathfrak{z}+ x \tau + y)&= e^{-i \pi t ( (x,x)\tau +2(x,
\mathfrak{z}) )} \varphi ( \tau, \mathfrak{z} ), \quad x,y \in L,
\end{align*}
and $\varphi$ has a Fourier expansion as 
\begin{equation*}
\varphi ( \tau, \mathfrak{z} )= \sum_{n\geq n_0 }\sum_{\ell\in L^\vee}f(n,
\ell)q^n\zeta^\ell,
\end{equation*}
where $n_0\in \ZZ$, $q=e^{2\pi i \tau}$ and $\zeta^\ell=e^{2\pi i (\ell,
\mathfrak{z})}$, then $\varphi$ is called a weakly holomorphic Jacobi form 
of weight $k$ and index $t$ associated to $L$. 
If $\varphi$ further satisfies the  condition
$( f(n,\ell) \neq 0 \Longrightarrow 2n - (\ell,\ell) \geq 0 )$
then $\varphi$ is called a holomorphic Jacobi form. If $\varphi$ further 
satisfies the stronger condition
$(f(n,\ell) \neq 0 \Longrightarrow 2n - (\ell,\ell) > 0 )$
then $\varphi$ is called a Jacobi cusp form. 
We denote by $J^{!}_{k,L,t}$ (resp. $J_{k,L,t}$, $J_{k,L,t}^{\text{cusp}}
$) the vector space of weakly holomorphic (resp. holomorphic, cusp) Jacobi 
forms of weight $k$ and index $t$ for $L$.
\end{definition}

The Jacobi forms in the sense of Eichler--Zagier \cite{EZ} are identical 
to the Jacobi forms $J_{k,A_1, t}$ for the lattice  $A_1=\latt{\ZZ, 2x^2}$ 
of rank $1$.

The Fourier coefficient $f(n,\ell)$ depends only on the number $2n-(\ell,
\ell)$ and the class of $\ell$ modulo $tL$.
The number $2n-(\ell,\ell)$ is called the hyperbolic norm of $f(n,\ell)$. 
The Fourier coefficients $f(n,\ell)$ with negative hyperbolic norm are 
called singular Fourier coefficients, which determine the divisor of 
Borcherds product.

\begin{theorem}[see Theorem 4.2 in \cite{G18} for details]
\label{th:Borcherds}
Let 
$$
\varphi(\tau,\mathfrak{z})=\sum_{n\in\ZZ, \ell\in L^\vee}f(n,\ell)q^n 
\zeta^\ell \in J^{!}_{0,L,1},
$$
and assume that $f(n,\ell)\in \ZZ$ for all $2n-(\ell,\ell)\leq 0$. We fix 
an ordering in the vector system $\{\ell;f(0,\ell)\}$ (see the bottom of 
page 825 in \cite{G18}). The notation $(n,\ell,m)>0$ means that either 
$m>0$, or $m=0$ and $n>0$, or $m=n=0$ and $\ell <0$. We set
\begin{align*}
&A=\frac{1}{24}\sum_{\ell\in L^\vee}f(0,\ell),& &\vec{B}=\frac{1}
{2}\sum_{\ell>0} f(0,\ell)\ell,& &C=\frac{1}{2\rank(L)}\sum_{\ell\in L^
\vee}f(0,\ell)(\ell,\ell).&
\end{align*}
Then the product
$$\Borch(\varphi)(Z)=q^A \zeta^{\vec{B}} \xi^C\prod_{\substack{n,m\in\ZZ, \ell
\in L^\vee\\ (n,\ell,m)>0}}(1-q^n \zeta^\ell \xi^m)^{f(nm,\ell)}, $$
where $Z= (\tau,\mathfrak{z}, \omega) \in \cH (L)$, $q=\exp(2\pi i \tau)$, 
$\zeta^\ell=\exp(2\pi i (\ell, \mathfrak{z}))$, $\xi=\exp(2\pi i \omega)$, 
defines a meromorphic modular form of weight $f(0,0)/2$ with respect to 
the stable orthogonal group $\widetilde{\Orth}^+(2U\oplus L(-1))$ with a 
character $\chi$ induced by
\begin{align*}
&\chi \lvert_{\SL_2(\ZZ)}=v_\eta^{24A},& &\chi \lvert_{H(L)}([\lambda,\mu; 
r])=e^{\pi i C((\lambda,\lambda)+(\mu, \mu )- (\lambda, \mu ) +2r)},& & 
\chi(V)=(-1)^D,&
\end{align*}
where $V: (\tau,\mathfrak{z}, \omega) \to (\omega,\mathfrak{z},\tau)$ and 
$D=\sum_{n<0}\sigma_0(-n) f(n,0)$. The poles and zeros of $\Borch(\varphi)
$ lie on the rational quadratic divisors $\cD_v$, where $v\in 2U\oplus L^
\vee(-1)$ is a primitive vector with $(v,v)<0$. The multiplicity of this 
divisor is given by 
$$ \mult \cD_v = \sum_{d\in \ZZ,d>0 } f(d^2n,d\ell),$$
where $n\in\ZZ$, $\ell\in L^\vee$ such that $(v,v)=2n-(\ell,\ell)$ and $v
\equiv \ell\mod 2U\oplus L(-1)$.
Moreover, the first Fourier--Jacobi coefficient of $\Borch(\varphi)$ is 
given by
\begin{equation}\label{FJtheta}
\psi_{L,C}(\tau,\mathfrak{z})=\eta(\tau)^{f(0,0)}\prod_{\ell 
>0}\left(\frac{\vartheta(\tau,(\ell,\mathfrak{z}))}{\eta(\tau)} 
\right)^{f(0,\ell)},
\end{equation}
which is a generalized theta block.
\end{theorem}

From the above theorem, we see that the Borcherds product is antisymmetric 
if the number $D$ is odd.

\section{Lifting scalar-valued modular forms to Jacobi forms}
\label{sec:lift}
In \cite{Sch06}, N. Scheithauer constructed a map which lifts scalar-valued 
modular forms on congruence subgroups to modular forms for the Weil 
representation.  In view of the isomorphism between modular forms for the 
Weil representation and Jacobi forms, we can easily build a lifting  from 
scalar-valued modular forms on congruence subgroups to Jacobi forms. This 
lifting plays a crucial role in this paper. 

For our purpose, we focus on lattices of prime level. Let $L$ be an even 
positive definite lattice with bilinear form $\latt{\cdot,\cdot}$. Denote 
the dual lattice of $L$ by $L^\vee$.
The level of $L$ is the smallest positive integer $N$ such that $N
\latt{x,x}\in 2\ZZ$ for all $x\in L^\vee$. We next assume that the level 
of $L$ is a prime number $p$. 
Let $D(L)=L^\vee/L$ be the discriminant group of $L$.  Let $\{e_\gamma: 
\gamma \in D(L)\}$ be the formal basis of the group ring $\CC[D(L)]$. We 
denote the Weil representation of $\SL_2(\ZZ)$ on $\CC[D(L)]$ by $
\rho_{D(L)}$ and the orthogonal group of $D(L)$ by $\Orth(D(L))$ (see e.g. \cite{Bo98, Sch09}). Let 
$M^{!,\inv}_k(\rho_{D(L)})$ be the space of nearly holomorphic modular 
forms for $\rho_{D(L)}$ of weight $k$ which are holomorphic except at 
infinity and invariant under the action of $\Orth(D(L))$ (see e.g. \cite{Bo98, Sch06}). 
By \cite[Theorem 6.2]{Sch06}, we have the following proposition.

\begin{proposition}\label{propsch}
Let $f\in M_{k}^!(\Gamma_0(p),\chi_{D(L)})$ be a scalar-valued nearly 
holomorphic modular form on $\Gamma_0(p)$ of weight $k$ and character 
$\chi_{D(L)}$ which is holomorphic except at cusps, where $\chi_{D(L)}$ is 
a Dirichlet character  defined as 
$$\chi_{D(L)}(A)=\left(\frac{a}{\lvert D(L)\rvert}\right),\quad A=
\left(\begin{array}{cc}
a & b \\ 
c & d
\end{array}  \right)\in \Gamma_0(p).$$ 
Then we have
\begin{equation}
F_{\Gamma_0(p),f,0}(\tau)= \sum_{M\in \Gamma_0(p) \backslash \SL_2(\ZZ)} f
\lvert_M(\tau) \rho_{D(L)}(M^{-1})e_0 \in M^{!,\inv}_k(\rho_{D(L)}).
\end{equation}
If we write 
$$ 
f\lvert_S (\tau)=\sum_{t=0}^{p-1}g_t(\tau),  \quad S=\left(\begin{array}
{cc}
0 & -1 \\ 
1 & 0
\end{array}  \right),
$$
where 
$$g_t(\tau+1)=\exp\left(\frac{2t\pi i}{p}\right) g_t(\tau), \quad 0\leq t 
\leq p-1,$$
then we have 
\begin{equation}\label{fs1}
F_{\Gamma_0(p),f,0}(\tau)=f(\tau)e_0+\xi_1\frac{p}{\sqrt{\lvert D(L)
\rvert}}\sum_{\gamma\in D(L)} g_{j_\gamma}(\tau) e_\gamma,
\end{equation}
here $j_\gamma/p = -\latt{\gamma,\gamma}/2 \mod 1$ for $\gamma\in D(L)$ 
and 
$$\xi_1=\left(\frac{-1}{\lvert D(L)\rvert}\right)\exp\left( \frac{\rank 
(L)\pi i}{4}\right).$$
\end{proposition}

We refer to \cite{Sch09, Sch15} for more properties of the above lifting and some other similar constructions of this type.

Recall that the theta functions for the lattice $L$ are defined as
\begin{equation}
\Theta_{\gamma}^{L}(\tau,\mathfrak{z})=\sum_{\ell \in \gamma+L}\exp
\left(\pi i\latt{\ell,\ell} \tau + 2\pi i\latt{\ell,\mathfrak{z}} \right), 
\quad \gamma\in D(L).
\end{equation}
By means of the isomorphism between vector-valued modular forms and Jacobi 
forms (for example, see \cite{CG}), we obtain the following result.

\begin{proposition}\label{proplift}
Under the assumptions of Proposition \ref{propsch}, if we write 
$$F_{\Gamma_0(p),f,0}(\tau)=\sum_{\gamma\in D(L)} F_{\Gamma_0(p),f,
0;\gamma}(\tau)e_\gamma,$$ then the function
\begin{equation}
\Psi_{\Gamma_0(p),f,0}(\tau,\mathfrak{z})=\sum_{\gamma\in D(L)} 
F_{\Gamma_0(p),f,0;\gamma}(\tau)\Theta_{\gamma}^{L}(\tau,\mathfrak{z})
\end{equation}
is a weakly holomorphic Jacobi form of weight $k+\frac{1}{2}\rank (L)$ and 
index $1$ for $L$ which is invariant under the action of the integral 
orthogonal group $\Orth(L)$.
\end{proposition}

\section{Antisymmetric paramodular forms of weight $3$ on $\Orth(2,8)$} \label{sec:wt3}
In this section we prove Theorem \ref{th:wt3}. The proof is based on 
Scheithauer's work on the classification of reflective modular forms of singular (i.e. minimal possible) weight.
By \cite[Theorem 10.3]{Sch06}, there exists  a holomorphic 
Borcherds product $\Phi_3^{\Sch}$ of 
singular weight $3$ with respect to the orthogonal group of the lattice
\begin{equation}\label{lat-BC}
U\oplus U(7)\oplus \text{Barnes-Craig lattice},
\end{equation}
whose genus is of type $\II_{2,8}(7^{-5})$. 
The modular form $\Phi_3^{\Sch}$ is a reflective 
modular form with complete 2-divisor and 14-divisor whose multiplicities 
are all one. 
Below we give another model of the lattice \eqref{lat-BC} and a new construction 
of the reflective modular form $\Phi_3^{\Sch}$.

Let  $A_6$ be the classical root lattice 
$$
A_6=\{(x_1,\dots,x_7)\in \ZZ^7: x_1+\dots+x_7=0\}.
$$
Following \cite{Bou60}, we fix the set of simple roots in $A_6$  
\begin{align*}
&\alpha_1=(1,-1,0,0,0,0,0)& &\alpha_2=(0,1,-1,0,0,0,0)& &
\alpha_3=(0,0,1,-1,0,0,0)&\\
&\alpha_4=(0,0,0,1,-1,0,0)& &\alpha_5=(0,0,0,0,1,-1,0)& &
\alpha_6=(0,0,0,0,0,1,-1).&
\end{align*}
Then the set of $21$ positive roots in $A_6$ is 
\begin{equation}\label{R-pos} 
R_2^+(A_6)=\left\{ \sum_{s=i}^j \alpha_s : 1\leq i \leq j \leq 6
\right\}.
\end{equation}
Let $w_i, 1\leq i \leq 6$ be the fundamental weights of $A_6$. Then $
(\alpha_i,w_j)=\delta_{ij}$ and 
$A_6^\vee/ A_6=\{0,w_1,w_2,w_3,w_4,w_5,w_6\}$. 
The level of $A_6$ is $7$. Thus, the 
renormalization
$$
A_6^\vee(7)=\ZZ w_1+\ZZ w_2+\ZZ w_3+\ZZ w_4+\ZZ w_5 +\ZZ w_6,\quad 
\latt{\cdot,\cdot}=7(\cdot,\cdot), 
$$
is an even integral lattice of determinant $7^5$ and its dual lattice is 
$(A_6^\vee(7))^\vee =\frac{1}{7}A_6$. Throughout this section, $(\cdot,
\cdot)$ denotes the standard scalar product on $\RR^6$.

By \cite[Corollary 1.13.3]{Nik80}, we have
\begin{equation}\label{BC-lattice}
U\oplus U(7)\oplus \text{Barnes-Craig lattice} \cong 2U\oplus A_6^\vee(7)
\end{equation}
because they are all of level $7$ and then belong to the same genus, thus   
to the same class.
We next use Proposition \ref{proplift} to construct the reflective
Borcherds product  $\Phi_3^{\Sch}$ at the $1$-dimensional cusp 
determined by the decomposition 
$2U\oplus A_6^\vee(-7)$. 

N. Scheithauer constructed a nearly holomorphic modular form of weight $-3$ for the Weil representation associated to the discriminant form of the lattice \eqref{BC-lattice} by Proposition \ref{propsch}.
The datum for it is a nearly holomorphic modular form 
$\eta^{-3}(\tau)\eta^{-3}(7\tau)$ of weight $-3$ and character 
$\left( \frac{\cdot}{7} \right)$ with respect to  $\Gamma_0(7)$.
By Proposition \ref{proplift}, we get a weakly holomorphic Jacobi form 
$\Psi_{A_6^\vee(7)}$ of weight $0$ and index $1$ for $A_6^\vee(7)$ which is invariant under the orthogonal group $\Orth(A_6^\vee(7))=\Orth(A_6)$.

\begin{theorem}\label{th:Schwt3}
The Borcherds product $\Phi_3^{\Sch}=\Borch(\Psi_{A_6^\vee(7)})$ is a reflective modular form of weight $3$ and character $\det$ for the group 
$\widetilde{\Orth}^+(2U\oplus A_6^\vee(-7))$. Its zero divisors are all simple and represented as  
\begin{equation}\label{eq:reflectdivisor}
\div (\Phi_3^{\Sch}) = \sum_{\substack{r\in 2U\oplus A_6^\vee(-7) \\ 
(r,r)_2=-2}} \cD_r + \sum_{\substack{s\in 2U\oplus \frac{1}{7} A_6(-1)\\ 
(s,s)_2=-\frac{2}{7}}} \cD_s,
\end{equation}
here $(\cdot,\cdot)_2$ is the bilinear form of the lattice $2U\oplus A_6^\vee(-7)$.
\end{theorem}

\begin{proof}
From the construction of the Jacobi form $\Psi_{A_6^\vee(7)}$, 
we see that its singular Fourier coefficients  are 
given by
$$
\sing(\Psi_{A_6^\vee(7)})=\sum_{n\in \NN} \sum_{\substack{r\in A_6^
\vee(7)\\ \latt{r,r}=2n}} q^{n-1}e^{2\pi i \latt{r,\mathfrak{z}}}+\sum_{n
\in \NN} \sum_{\substack{s\in \frac{1}{7}A_6\\ \latt{s,s}=2n+\frac{2}{7}}} 
q^{n}e^{2\pi i \latt{s,\mathfrak{z}}}.
$$
A Fourier coefficient depends only on the hyperbolic norm of its index
and the class of $\ell$ in the discriminant group. In particular,
all Fourier coefficients in $q^0$-term of $\Psi_{A_6^\vee(7)}$ are singular except the constant term $f(0,0)=6$. Thus we have
\begin{equation}\label{eq:q^0}
\begin{split}
\Psi_{A_6^\vee(7)}(\tau,\mathfrak{z})=&\Psi_{\Gamma_0(7),\eta^{-3}
(\tau)\eta^{-3}(7\tau),0}\\
=&q^{-1}+\sum_{\substack{r\in A_6\\ (r,r)=2}}e^{2\pi i(r,\mathfrak{z})}
+6+O(q) \in J_{0,A_6^\vee(7),1}^{!,\Orth(A_6)},
\end{split}
\end{equation}
where $\mathfrak{z}=\sum_{i=1}^6w_iz_i$, $z_i\in \CC$.
By \cite[Proposition 3.2]{Sch06}, there are $2352$ classes of norm $
\frac{2}{7}$ ($\m 2 \ZZ$) in the discriminant group of $A_6^\vee(7)$. But 
we can only see $42$ of them from the $q^0$-term in the Fourier expansion 
of  $\Psi_{A_6^\vee(7)}$. 

According to Theorem \ref{th:Borcherds} and the Eichler criterion 
(see \cite{GHS1}),  the automorphic product 
$\Borch(\Psi_{A_6^\vee(7)})$  and $\Phi_3^{\Sch}$ 
have the same divisor \eqref{eq:reflectdivisor} with respect to the modular group 
$\widetilde{\Orth}^+(2U\oplus A_6^\vee(-7))$. Therefore, the functions are 
equal, up to a constant, due to the K\"ocher principle. To see that the 
constant is one,
one can use the fact that  both automorphic products are constructed 
by the same modular form $\eta^{-3}(\tau)\eta^{-3}(7\tau)$.

The lattice $2U\oplus A_6^\vee(-7)$ satisfies the Kneser condition (see 
\cite{GHS1}).  Therefore the unique nontrivial 
character of $\widetilde{\Orth}^+(2U\oplus A_6^\vee(-7))$ 
is $\det$ (see \cite[Corollary 1.8, 
Proposition 3.4]{GHS1}). Thus the modular form 
$\Borch(\Psi_{A_6^\vee(7)})$ has character $\det$  because it is  
antisymmetric. 
\end{proof}

The advantage of our description of Scheithauer's form $\Phi_3^{\Sch}$
at the one-dimensional cusp related to $2U\oplus A_6^\vee(-7)$
is that we can give an explicit formula for its first Fourier--Jacobi 
coefficient.

\begin{corollary}\label{coro:Schwt3}
The first Fourier--Jacobi coefficient of 
$\Phi_3^{\Sch}$ is a holomorphic Jacobi form defined by the following theta block
\begin{equation}
\begin{split}
&\frac{1}{\eta^{15}(\tau)}\prod_{r\in R_2^{+}(A_6)}
\vartheta(\tau,(r,\mathfrak{z}))\\
=&
\vartheta(z_1)\vartheta(z_2)\vartheta(z_3)\vartheta(z_4)\vartheta(z_5)
\vartheta(z_6)\vartheta(z_1+z_2)\vartheta(z_2+z_3)\vartheta(z_3+z_4)\\
&
\vartheta(z_4+z_5)\vartheta(z_5+z_6)\vartheta(z_1+z_2+z_3)\vartheta(z_2+z_
3+z_4)\vartheta(z_3+z_4+z_5)\\
&
\vartheta(z_4+z_5+z_6)\vartheta(z_1+z_2+z_3+z_4)\vartheta(z_2+z_3+z_4+z_5)
\\
&\vartheta(z_3+z_4+z_5+z_6)\vartheta(z_1+z_2+z_3+z_4+z_5)\\
&\vartheta(z_2+z_3+z_4+z_5+z_6)\vartheta(z_1+z_2+z_3+z_4+z_5+z_6)/
\eta^{15},
\end{split}
\end{equation}
where $R_2^{+}(A_6)$ is the set of $21$ positive roots of $A_6$
(see \eqref{R-pos} and \eqref{KWD})
and $\vartheta(z)=\vartheta(\tau,z)$. It is a holomorphic Jacobi 
form of singular weight $3$ and index $1$ for $A_6^\vee(7)$ which is identical to the Kac--Weyl denominator function of the affine Lie algebra 
$\hat{\mathfrak g}(A_6)$ (see \cite[Corollary 2.7]{G18}). 
\end{corollary}

\begin{proof}
According to Theorem \ref{th:Borcherds}, to write the 
first Fourier--Jacobi coefficient of the Borcherds product 
$\Borch(\Psi_{A_6^\vee(7)})$ we need to know
only $q^0$-part of the Fourier expansion of $\Psi_{A_6^\vee(7)}$. Thus we finish the proof by \eqref{eq:q^0}.
\end{proof}

\begin{remark}
In fact, $\Phi_3^{\Sch}$ is a modular form for the full modular group $\Orth^+(2U\oplus A_6^\vee(-7))$ because the vector-valued modular form $F_{\Gamma_0(7),\eta^{-3}(\tau)\eta^{-3}(7\tau),0}$ is invariant under the orthogonal group of the discriminant form of $2U\oplus A_6^\vee(-7)$ (see \cite{Sch06}). It would be interesting to describe the character of $\Phi_3^{\Sch}$ for $\Orth^+(2U\oplus A_6^\vee(-7))$.
\end{remark}

We next consider the quasi pull-back of the Borcherds product $\Phi_3^{\Sch}$
(see \cite{GHS2} or \cite{G18}) to complete the proof of Theorem \ref{th:wt3}. In our case we can make using  pull-backs
of the Jacobi modular form $\Psi_{A_6^\vee(7)}$.  Given $\mathbf{a}
=(a_1,a_2,a_3,a_4,a_5,a_6)\in \ZZ^6$, we define a Jacobi form in one variable 
\begin{equation}
\Psi_{A_6^\vee(7),\mathbf{a}}(\tau,z)=\Psi_{A_6^\vee(7)}\left(\tau,z
\sum_{i=1}^6a_iw_i\right).
\end{equation}
We denote by $n_0(\mathbf{a})$ the number of $0$ in the following $21$ 
integers
\begin{align*}
&a_1, a_2, a_3, a_4, a_5, a_6, a_1+a_2, a_2+a_3, a_3+a_4, a_4+a_5, 
a_5+a_6,\\
&a_1+a_2+a_3, a_2+a_3+a_4, a_3+a_4+a_5, a_4+a_5+a_6,\\
&a_1+a_2+a_3+a_4, a_2+a_3+a_4+a_5, a_3+a_4+a_5+a_6,\\
&a_1+a_2+a_3+a_4+a_5, a_2+a_3+a_4+a_5+a_6,\\
&a_1+a_2+a_3+a_4+a_5+a_6.
\end{align*}
We also set 
\begin{equation}
N(\mathbf{a})=\frac{7}{2}\left(\sum_{i=1}^6a_iw_i, \sum_{i=1}^6a_iw_i 
\right),
\end{equation}
which equals the half of the sum of the squares of the above 21 integers. 
The explicit formula of $N(\mathbf{a})$ is given in \eqref{N(a)}.
Then the function $\Borch(\Psi_{A_6^\vee(7),\mathbf{a}})$ is an 
antisymmetric holomorphic Siegel modular form of weight $3+n_0(\mathbf{a})
$ with respect to the paramodular group of level $N(\mathbf{a})$. The 
theta block \eqref{TB3} is not identically zero if and only if 
$n_0(\mathbf{a})=0$. 

To finish the proof of Theorem \ref{th:wt3} we have to apply 
the cuspidality test.

\begin{proposition}[Proposition 3.1 in \cite{PSY}] \label{prop:test}
Let $t$ be a square-free positive integer, and let $k$ be a positive 
integer. If $k = 2$ or $k$ is odd then $M_k(\Gamma_t)=S_k(\Gamma_t)$. If 
$k = 4$, $6$, $8$, $10$, $14$ then
for all $F \in M_k (\Gamma_t)$, $F\in S_k (\Gamma_t)$ if and only if 
$c(0,0,0)= 0$ in \eqref{FJ}.
\end{proposition}

Applying Theorem \ref{th:wt3} to different $\mathbf{a}$, we can construct
the infinite series of  antisymmetric paramodular forms of weight $3$. The first six values of $N(\mathbf{a})$ in Theorem \ref{th:wt3} are
\begin{align*}
&98: \mathbf{a}=(1,1,1,1,1,1),& &122: \mathbf{a}=(2,1,1,1,1,1),& &138: \mathbf{a}=(1,2,1,1,1,1),\\
&146: \mathbf{a}=(1,1,2,1,1,1),& &147: \mathbf{a}=(-1,4,-6,4,1,3),& &152: \mathbf{a}=(-1,2,1,2,-1,2).
\end{align*} 
The paramodular form for $t=98$ is not a cusp form  because its first Fourier--Jacobi coefficient 
$
\vartheta^6\vartheta_2^5\vartheta_3^4\vartheta_4^3\vartheta_5^2
\vartheta_6/\eta^{15}
$ 
is not a Jacobi cusp form. 

\begin{table}[ht]
\caption{Antisymmetric paramodular cusp forms of weight 3 and prime level 
$<300$}
\label{tableanti3prime}
\renewcommand\arraystretch{1.5}
\noindent\[
\begin{array}{|c|c|c|}
\hline 
N(\mathbf{a}) & \mathbf{a}=(a_1,...,a_6) & \text{Theta block} \\ 
\hline 
167 & (1,1,1,1,2,2) & 
\vartheta^4\vartheta_2^5\vartheta_3^3\vartheta_4^3\vartheta_5^2
\vartheta_6^2\vartheta_7\vartheta_8/\eta^{15}   \\
\hline 
173 & (1,1,2,1,1,2) & 
\vartheta^4\vartheta_2^4\vartheta_3^3\vartheta_4^4\vartheta_5^2
\vartheta_6^2\vartheta_7\vartheta_8/\eta^{15} \\
\hline 
223 & (-2,4,-7,6,3,-8) & 
\vartheta^3\vartheta_2^4\vartheta_3^3\vartheta_4^3\vartheta_5^2
\vartheta_6^3\vartheta_7\vartheta_8\vartheta_9/\eta^{15}  \\ 
\hline 
227 & (-3,-2,3,4,2,-8) & 
\vartheta^3\vartheta_2^5\vartheta_3^2\vartheta_4^3\vartheta_5^2
\vartheta_6^2\vartheta_7^2\vartheta_8\vartheta_9/\eta^{15} \\ 
& (2,3,-4,-2,-4,2) & 
\vartheta^3\vartheta_2^4\vartheta_3^3\vartheta_4^3\vartheta_5^3
\vartheta_6^2\vartheta_7\vartheta_8\vartheta_{10}/\eta^{15}   \\ 
\hline 
251 & (-6,4,1,3,-5,-5) & 
\vartheta^3\vartheta_2^4\vartheta_3^3\vartheta_4^2\vartheta_5^3
\vartheta_6^2\vartheta_7\vartheta_8^2\vartheta_{10}/\eta^{15}   \\ 
\hline 
257 & (8,-4,-1,3,4,-8) & 
\vartheta^2\vartheta_2^5\vartheta_3^2\vartheta_4^4\vartheta_5\vartheta_6^3
\vartheta_7\vartheta_8^2\vartheta_{10}/\eta^{15}   \\
& (7,-3,-5,7,-8,4) & 
\vartheta^3\vartheta_2^4\vartheta_3^2\vartheta_4^3\vartheta_5^2
\vartheta_6^2\vartheta_7^2\vartheta_8^2\vartheta_{9}/\eta^{15} \\
\hline 
269 & (6,-8,3,-5,6,3) & 
\vartheta^3\vartheta_2^3\vartheta_3^2\vartheta_4^4\vartheta_5^3
\vartheta_6^2\vartheta_7\vartheta_8\vartheta_{9}\vartheta_{10}/\eta^{15}\\
&(3,-8,1,3,2,-8) & 
\vartheta^3\vartheta_2^3\vartheta_3^3\vartheta_4^3\vartheta_5^2
\vartheta_6^2\vartheta_7^2\vartheta_8^2\vartheta_{10}/\eta^{15}\\
\hline 
271 & (4,-5,3,6,-3,-7) & 
\vartheta^3\vartheta_2^3\vartheta_3^3\vartheta_4^3\vartheta_5^2
\vartheta_6^3\vartheta_7\vartheta_8\vartheta_{9}\vartheta_{10}/\eta^{15} 
\\
\hline 
283 & (5,-6,-2,-3,9,-2) & 
\vartheta^2\vartheta_2^4\vartheta_3^3\vartheta_4^3\vartheta_5^2
\vartheta_6^3\vartheta_7\vartheta_8\vartheta_{9}\vartheta_{11}/\eta^{15}  
\\
& (-6,3,4,-6,-2,-2) & 
\vartheta^3\vartheta_2^3\vartheta_3^3\vartheta_4^3\vartheta_5\vartheta_6^3
\vartheta_7^2\vartheta_8\vartheta_{9}\vartheta_{10}/\eta^{15}  \\
\hline 
293 & (-8,-2,3,-4,5,2) & 
\vartheta^3\vartheta_2^3\vartheta_3^3\vartheta_4^4\vartheta_5\vartheta_6^2
\vartheta_7^2\vartheta_8\vartheta_{10}\vartheta_{11}/\eta^{15}  \\
& (6,-9,5,-4,-2,8) & 
\vartheta^2\vartheta_2^5\vartheta_3\vartheta_4^4\vartheta_5\vartheta_6^3
\vartheta_7\vartheta_8^2\vartheta_{9}\vartheta_{10}/\eta^{15}  \\
& (-6,2,7,9,-6,1) & 
\vartheta^2\vartheta_2^4\vartheta_3^2\vartheta_4^4\vartheta_5^2
\vartheta_6^2\vartheta_7^2
\vartheta_8\vartheta_{9}\vartheta_{11}/\eta^{15}  \\
& (1,-6,1,-5,2,6) & 
\vartheta^3\vartheta_2^3\vartheta_3^2\vartheta_4^3\vartheta_5^3
\vartheta_6^2\vartheta_7
\vartheta_8^2\vartheta_{9}\vartheta_{10}/\eta^{15}\\
\hline 
\end{array} 
\]
\end{table}

The other levels are listed in Tables \ref{tableanti3prime}, \ref{tableanti3(1)}, \ref{tableanti3(2)}.
We explain how to read the tables. For a fixed row, the corresponding paramodular form is constructed as 
$\Borch(\Psi_{A_6^\vee(7),\mathbf{a}})$. The number $N(\mathbf{a})$ is the 
level $t$ of the corresponding paramodular group. ``Theta block'' is the 
first Fourier--Jacobi coefficient of 
$\Borch(\Psi_{A_6^\vee(7), \mathbf{a}})$. 
 
The first antisymmetric paramodular cusp form of weight $3$ we know of is 
in $S_3(\Gamma_{122}^+)$. We note that  this function is not in 
$S_3(\Gamma_{122}^*)$ and it has Atkin-Lehner signs of $-1$ at both $2$ 
and $61$ (see \cite{GPY2}). These tables show that we can reconstruct all 
antisymmetric paramodular cusp forms of weight $3$ and square-free level 
$t$ from \cite{GPY2} except $t=197$. 

Jerry Shurman informed us that he can prove the nonexistence of antisymmetric paramodular forms for many $t<300$. For example,  
for square-free $t\le 220$ the space $S_3(\Gamma_t^+)$ might be nontrivial only for
$$
t=122, 138, 146, \bold {158}, 167, \bold {170}, 173, 174,
178, 182, \bold {183}, \bold {186}, 194, \bold{197}, 
$$
$$ 
202, 203, 206, 210, 213, 215, \bold {218}, \bold {219}.
$$
In bold face we write the polarisations for which we cannot construct 
an antisymmetric paramodular form of weight $3$. 

\newpage

\begin{table}[ht]
\caption{Antisymmetric paramodular cusp forms of weight 3 and squarefree 
(non prime) level $<300$ I}
\label{tableanti3(1)}
\renewcommand\arraystretch{1.5}
\noindent\[
\begin{array}{|c|c|c|}
\hline 
N(\mathbf{a}) & \mathbf{a}=(a_1,...,a_6) & \text{Theta block} \\ 
\hline 
122 & (2,1,1,1,1,1) & \vartheta^5\vartheta_2^5\vartheta_3^4\vartheta_4^3\vartheta_5^2\vartheta_6\vartheta_7/\eta^{15}  \\ 
\hline 
138 & (1,2,1,1,1,1) & \vartheta^5\vartheta_2^4\vartheta_3^4\vartheta_4^3\vartheta_5^2\vartheta_6^2\vartheta_7/\eta^{15}   \\ 
\hline 
146 & (1,1,2,1,1,1) & \vartheta^5\vartheta_2^4\vartheta_3^3\vartheta_4^3\vartheta_5^3\vartheta_6^2\vartheta_7/\eta^{15}   \\ 
\hline 
174 & (-1,4,1,-6,3,5) & \vartheta^4\vartheta_2^4\vartheta_3^4\vartheta_4^2\vartheta_5^3\vartheta_6^2\vartheta_7
\vartheta_8/\eta^{15}  \\
\hline 
178 & (1,3,-2,-4,9,-4) & \vartheta^4\vartheta_2^4\vartheta_3^4\vartheta_4^3\vartheta_5^2\vartheta_6^2\vartheta_7
\vartheta_9/\eta^{15}  \\
\hline 
182 & (7,-4,-1,4,-2,-5) & \vartheta^4\vartheta_2^3\vartheta_3^4\vartheta_4^4\vartheta_5^2\vartheta_6\vartheta_7^2
\vartheta_8/\eta^{15} \\
\hline 
194 & (2,5,-2,-2,-5,6) & \vartheta^3\vartheta_2^5\vartheta_3^3\vartheta_4^3\vartheta_5^3\vartheta_6\vartheta_7^2
\vartheta_9/\eta^{15}  \\

 & (-1,3,-5,8,-7,6) & \vartheta^4\vartheta_2^4\vartheta_3^3\vartheta_4^2\vartheta_5^3\vartheta_6^2\vartheta_7^2
\vartheta_8/\eta^{15}  \\
\hline 

202&(5,-2,-2,-2,3,6)&\vartheta^4\vartheta_2^4\vartheta_3^4\vartheta_4^2\vartheta_5^2\vartheta_6^2\vartheta_7
\vartheta_8\vartheta_9/\eta^{15}  \\

&(-3,-2,7,-6,3,-6)&\vartheta^3\vartheta_2^4\vartheta_3^4\vartheta_4^3\vartheta_5^2\vartheta_6^2\vartheta_7^2
\vartheta_9/\eta^{15}  \\
\hline

203 & (2,2,4,-1,-1,-5) & \vartheta^4\vartheta_2^4\vartheta_3^2\vartheta_4^3\vartheta_5^2\vartheta_6^3\vartheta_7
^2\vartheta_8/\eta^{15}  \\
\hline 

206 & (-6,2,7,-8,4,-1) & \vartheta^4\vartheta_2^3\vartheta_3^3\vartheta_4^4\vartheta_5^3\vartheta_6\vartheta_7
\vartheta_8\vartheta_9/\eta^{15}  \\

& (-5,4,-6,1,7,-5)& \vartheta^5\vartheta_2^3\vartheta_3^2\vartheta_4^2\vartheta_5^3\vartheta_6^3\vartheta_7
^2\vartheta_8/\eta^{15}  \\
\hline  
210& (-7,-3,7,-3,2,-1)& \vartheta^4\vartheta_2^3\vartheta_3^4\vartheta_4^3\vartheta_5^2\vartheta_6^2\vartheta_7
^2\vartheta_{10}/\eta^{15}  \\
\hline 

213 & (9,-5,2,2,-3,-3) & \vartheta^3\vartheta_2^4\vartheta_3^3\vartheta_4^4\vartheta_5^2\vartheta_6^2\vartheta_7
\vartheta_8\vartheta_9/\eta^{15}  \\
\hline 

215& (-3,-7,8,-3,1,-2)& \vartheta^4\vartheta_2^4\vartheta_3^3\vartheta_4^3\vartheta_5^2\vartheta_6^2\vartheta_7
\vartheta_8\vartheta_{10}/\eta^{15}  \\
\hline 

222& (-3,7,-2,1,-2,-6) &\vartheta^4\vartheta_2^4\vartheta_3^3\vartheta_4^2\vartheta_5^2\vartheta_6^2\vartheta_7^2
\vartheta_8\vartheta_9/\eta^{15}  \\

 &(-2,-7,1,5,-2,-1) & \vartheta^3\vartheta_2^3\vartheta_3^5\vartheta_4^2\vartheta_5^2\vartheta_6^3\vartheta_7
\vartheta_8\vartheta_9/\eta^{15}  \\
\hline  

230 &(7,-2,-2,7,-6,-2)&\vartheta^3\vartheta_2^4\vartheta_3^4\vartheta_4^2\vartheta_5^3\vartheta_6\vartheta_7^2
\vartheta_8\vartheta_{10}/\eta^{15}  \\

 &(1,-3,1,-2,-1,-5)&\vartheta^5\vartheta_2^4\vartheta_3^3\vartheta_4^2\vartheta_5^2\vartheta_6\vartheta_7
\vartheta_8\vartheta_9\vartheta_{10}/\eta^{15}  \\
\hline 

237&(2,2,-6,3,-4,-3)& \vartheta^3\vartheta_2^3\vartheta_3^4\vartheta_4^4\vartheta_5\vartheta_6^2\vartheta_7^2
\vartheta_8\vartheta_{10}/\eta^{15}  \\
\hline 

238&(-7,2,-1,8,-3,-3)&\vartheta^4\vartheta_2^3\vartheta_3^3\vartheta_4^2\vartheta_5^2\vartheta_6^3\vartheta_7^2
\vartheta_8\vartheta_9/\eta^{15}  \\
\hline 

255&(-3,-3,1,-3,6,-8)&\vartheta^2\vartheta_2^4\vartheta_3^4\vartheta_4^2\vartheta_5^3\vartheta_6^2\vartheta_7
\vartheta_8^2\vartheta_{10}/\eta^{15}  \\
\hline 

258&(-7,4,-6,-1,6,-2)&\vartheta^3\vartheta_2^2\vartheta_3^5\vartheta_4^3\vartheta_5\vartheta_6^3\vartheta_7^2
\vartheta_9\vartheta_{10}/\eta^{15}  \\

&(-5,1,2,-4,3,-8)&\vartheta^4\vartheta_2^4\vartheta_3^3\vartheta_4^2\vartheta_5^2\vartheta_6^2\vartheta_7
\vartheta_8\vartheta_9\vartheta_{11}/\eta^{15}  \\

&(6,-9,7,-3,-3,7)&\vartheta^4\vartheta_2^2\vartheta_3^3\vartheta_4^3\vartheta_5^2\vartheta_6^2\vartheta_7^2
\vartheta_8^2\vartheta_9/\eta^{15}  \\

&(-2,-3,1,9,-4,1)& \vartheta^3\vartheta_2^3\vartheta_3^3\vartheta_4^3\vartheta_5^3\vartheta_6^2\vartheta_7^2
\vartheta_9\vartheta_{10}/\eta^{15}  \\
\hline 
\end{array} 
\]
\end{table}

\newpage 

\begin{table}[ht]
\caption{Antisymmetric paramodular cusp forms of weight 3 and squarefree (non prime) level $<300$ II}
\label{tableanti3(2)}
\renewcommand\arraystretch{1.5}
\noindent\[
\begin{array}{|c|c|c|}
\hline 
N(\mathbf{a}) & \mathbf{a}=(a_1,...,a_6) & \text{Theta block} \\ 
\hline 

262& (-5,3,-5,-1,2,7)&\vartheta^4\vartheta_2^3\vartheta_3^3\vartheta_4\vartheta_5^2\vartheta_6^3\vartheta_7^2
\vartheta_8^2\vartheta_9/\eta^{15}  \\

&(-1,-1,7,3,-5,-1)&\vartheta^3\vartheta_2^3\vartheta_3^4\vartheta_4^2\vartheta_5^3\vartheta_6^2\vartheta_7
\vartheta_8\vartheta_9\vartheta_{10}/\eta^{15}  \\

&(7,-6,-3,1,5,-7)&
\vartheta^4\vartheta_2^3\vartheta_3^4\vartheta_4^2\vartheta_5\vartheta_6^2\vartheta_7^2
\vartheta_8\vartheta_9\vartheta_{10}/\eta^{15}  \\
\hline 

266& (-1,-7,3,6,-4,5)&\vartheta^3\vartheta_2^4\vartheta_3^3\vartheta_4^2\vartheta_5^3\vartheta_6\vartheta_7^2
\vartheta_8\vartheta_9\vartheta_{10}/\eta^{15}  \\

&(4,-3,-6,7,-3,7)&\vartheta^3\vartheta_2^4\vartheta_3^2\vartheta_4^3\vartheta_5^3\vartheta_6^2\vartheta_7^2
\vartheta_9\vartheta_{11}/\eta^{15}  \\

&(8,-1,-4,2,5,-8)&\vartheta^2\vartheta_2^4\vartheta_3^4\vartheta_4\vartheta_5^4\vartheta_6\vartheta_7^2
\vartheta_8^2\vartheta_{10}/\eta^{15}  \\
\hline 

278 &(4,1,2,3,-1,-7)&\vartheta^2\vartheta_2^4\vartheta_3^3\vartheta_4^2\vartheta_5^4\vartheta_6\vartheta_7^2
\vartheta_8\vartheta_9\vartheta_{10}/\eta^{15}  \\

&(5,-7,6,-5,8,-1)&
\vartheta^5\vartheta_2^3\vartheta_3\vartheta_4\vartheta_5^2\vartheta_6^3\vartheta_7^3
\vartheta_8^2\vartheta_9/\eta^{15}  \\

&(-8,1,3,-5,4,-6)&\vartheta^3\vartheta_2^3\vartheta_3^3\vartheta_4^4\vartheta_5^2\vartheta_6\vartheta_7^2
\vartheta_8\vartheta_9\vartheta_{11}/\eta^{15}  \\

&(-2,7,-8,7,-9,4)&\vartheta^4\vartheta_2^3\vartheta_3^2\vartheta_4^2\vartheta_5^3\vartheta_6^2\vartheta_7^2
\vartheta_8\vartheta_9\vartheta_{10}/\eta^{15}  \\

&(1,-2,-9,7,1,-2)&
\vartheta^5\vartheta_2^4\vartheta_3^3\vartheta_4^2\vartheta_5\vartheta_6\vartheta_7
\vartheta_8\vartheta_9\vartheta_{10}\vartheta_{11}/\eta^{15}  \\

&(-6,-4,-1,4,2,-3)&
\vartheta^4\vartheta_2^3\vartheta_3^3\vartheta_4^2\vartheta_5^3\vartheta_6^2\vartheta_7
\vartheta_8\vartheta_{10}\vartheta_{11}/\eta^{15}  \\
\hline 

285& (-4,-2,8,-5,-5,1)&
\vartheta^3\vartheta_2^3\vartheta_3^3\vartheta_4^3\vartheta_5^2\vartheta_6^2\vartheta_7
\vartheta_8^2\vartheta_9\vartheta_{10}/\eta^{15}  \\
\hline

286&(-1,-3,6,3,-4,-7)&\vartheta^3\vartheta_2^3\vartheta_3^3\vartheta_4^2\vartheta_5^3\vartheta_6^3\vartheta_7
\vartheta_8\vartheta_9\vartheta_{11}/\eta^{15}  \\
\hline

287&(2,2,-3,-5,7,-9)&\vartheta^4\vartheta_2^4\vartheta_3^2\vartheta_4^2\vartheta_5\vartheta_6^2\vartheta_7^2
\vartheta_8^2\vartheta_9\vartheta_{10}/\eta^{15}  \\

&(8,4,-6,1,-2,-1)&\vartheta^4\vartheta_2^3\vartheta_3^2\vartheta_4^3\vartheta_5^2\vartheta_6^2\vartheta_7^2
\vartheta_8^2\vartheta_{12}/\eta^{15}  \\
\hline

290&(-2,-7,5,5,-8,6)&\vartheta^3\vartheta_2^4\vartheta_3^3\vartheta_4\vartheta_5^3\vartheta_6\vartheta_7^2
\vartheta_8^2\vartheta_9\vartheta_{10}/\eta^{15}  \\

&(-4,-2,1,3,1,6)&\vartheta^4\vartheta_2^3\vartheta_3^2\vartheta_4^3\vartheta_5^3\vartheta_6^2\vartheta_7
\vartheta_9\vartheta_{10}\vartheta_{11}/\eta^{15}  \\

&(1,5,5,-9,2,-1)&\vartheta^4\vartheta_2^4\vartheta_3^3\vartheta_4^2\vartheta_5^2\vartheta_6\vartheta_7
\vartheta_8\vartheta_9\vartheta_{10}\vartheta_{11}/\eta^{15}  \\

&(-2,-5,1,7,2,-7)&\vartheta^2\vartheta_2^4\vartheta_3^3\vartheta_4^2\vartheta_5^3\vartheta_6\vartheta_7^3
\vartheta_8\vartheta_9\vartheta_{10}/\eta^{15}  \\

&(3,-4,7,-2,6,-5)&\vartheta^4\vartheta_2^2\vartheta_3^2\vartheta_4^3\vartheta_5^3\vartheta_6^3\vartheta_7^2
\vartheta_{10}\vartheta_{11}/\eta^{15}  \\
\hline

299&(5,1,4,-2,3,-7)&\vartheta^3\vartheta_2^3\vartheta_3^2\vartheta_4^3\vartheta_5^3\vartheta_6^3\vartheta_7
\vartheta_8\vartheta_{10}\vartheta_{11}/\eta^{15}  \\

&(5,4,-8,6,-5,-3)&\vartheta^3\vartheta_2^4\vartheta_3^2\vartheta_4^2\vartheta_5^2\vartheta_6^2\vartheta_7^2
\vartheta_8^2\vartheta_9\vartheta_{10}/\eta^{15}  \\
\hline
\end{array} 
\]
\end{table}

\newpage

\begin{remark}
The theta function $\theta_{A_6^\vee(7)}(\tau)=\sum_{l\in A_6^\vee(7)}
\exp(\pi i \latt{l,l}\tau)$ is a scalar-valued nearly holomorphic modular 
form on $\Gamma_0(7)$ of weight $3$ with the character $(\frac{\cdot}{7})
$. It can be expressed in terms of Dedekind $\eta$-functions
\begin{align*}
\theta_{A_6^\vee(7)}(\tau)&=\frac{\eta^7(\tau)}{\eta(7\tau)}
+7\eta^3(\tau)
\eta^3(7\tau)+7\frac{\eta^7(7\tau)}{\eta(\tau)}\\
&=1+14q^3+42q^5+70q^6+\cdots.
\end{align*}
\end{remark}

\begin{remark}
By Theorem \ref{th:wt3} we get a holomorphic Borcherds product. 
Therefore, its first Fourier--Jacobi coefficient is also holomorphic. It 
gives a new ``Borcherds type" proof of the holomorphicity
of the theta blocks of type $\frac{21-\vartheta}{15-\eta}$.

In \eqref{KWD} we give the description of the Kac--Weyl denominator function 
of an affine Lie algebra in terms of Jacobi theta-functions.
The fact that this Jacobi form is  holomorphic follows form the structure theory of affine Lie algebras or from so-called Mackdonalds identities.
A new purely arithmetic proof is given in \cite{GSZ}.
\end{remark}

\section{Applications}\label{sec:6}

\subsection{Applications to the theory of moduli spaces and group cohomology}
The paramodular group $\Gamma_t$ and 
its normal extensions in $\Sp_2(\RR)$ have realisations  as integral orthogonal groups of signature $(2,3)$. This realisation describes 
the nature of the normal extensions $\Gamma_t^+$ and $\Gamma_t^*$
(see \cite{GH98}).

Let $L_t=2U\oplus \latt{-2t}$ be an even integral lattice of signature 
$(2,3)$. 
The finite discriminant group
$$
A_{t}=L_t^\vee/L_t = (2t)^{-1} \ZZ/\ZZ
$$
is a finite abelian group equipped with a quadratic form
$$
q_t:\, A_t\times A_t\to (2t)^{-1} \ZZ/2\ZZ,\qquad
q_t(l,l)\equiv (l,l)_{L_t}\text{\,mod\,}2\ZZ
$$
(see \cite{Nik80} for a general definition).
Any $g\in \Orth(L_t)$ acts on the  finite group $A_t$.
By 
$\widetilde{\Orth}(L_t)$
we denote the subgroup of the orthogonal group consisting of elements
which act identically on the discriminant group.

The natural  projection of  ${\Orth^+}(L_t)$ on the finite  orthogonal group $\Orth(A_t)$ is surjective. The last group 
can be described as follows.
For every $d||t$ (i.e. $d|t$ and $(d,\frac{t}d)=1$)
there exists a unique  ($\m 2t$) integer $\xi_d$
satisfying
$$
\xi_d=-1\, \operatorname{mod }\, 2 d,\quad \xi_d=
1\, \operatorname{mod }\, 2t/d.
$$
All such  $\xi_d$ form the group
$$
\Xi (t)=\{\,\xi \operatorname{mod} 2 t\ |\
\xi^2= 1 \operatorname{mod }4 t\,\}
\cong (\ZZ/2\ZZ)^{\nu (t)},                     
$$
where $\nu(t)$ is the number of prime divisors of $t$. It is evident
that
$\Orth(A_t)\cong \Xi (t)$.

According to \cite{G94} and 
\cite[Proposition 1.2 and Corollary 1.3]{GH98}) 
we have the following isomorphisms  
$\Gamma_t/\{\pm E_4\}\cong \widetilde{\SO}^+(L_t)$
and 
$$
\Gamma_t^+/\{\pm E_4\}\cong \widetilde{\Orth}^+(L_t)/\{\pm E_5\},\quad
\Gamma_t^*/\{\pm E_4\}\cong {\Orth}^+(L_t)/\{\pm E_5\}.
$$
The coverings $\Gamma_t\setminus \HH_2\to \Gamma_t^+\setminus \HH_2$
and $\Gamma_t\setminus \HH_2\to \Gamma_t^*\setminus \HH_2$
are galois with a finite abelian Galois group.
According to \cite[Proposition 1.5]{GH98}, 
the modular variety  
${\mathcal A}_t^+=\Gamma_t^+\setminus \HH_2$ ($t$ is square-free)
is isomorphic to {\it the moduli space of polarized $K3$ surfaces with
a polarisation of type} $\latt{2t}\oplus\, 2E_8(-1)$.
According to \cite[Theorem 1.5]{GH98}, 
the modular variety 
${\mathcal K}_t=\Gamma_t^*\setminus \HH_2$
is isomorphic to {\it the moduli space of
Kummer surfaces associated to abelian surfaces with a 
$(1,t)$-polarisation}.

We mentioned in the introduction that  weight $3$ cusp forms 
are closely related to canonical differential forms on smooth models of 
the corresponding modular
varieties. If $F$ is a cusp form of weight $3$ with respect to an arithmetic  
group $\Gamma$, then $\omega_F=F(Z)dZ$ is a holomorphic $3$-form 
over the open smooth part of the modular variety 
$(\Gamma\setminus \HH_2)^o$
outside the ramification divisor and the boundary components.
A very useful extension theorem due to E. Freitag implies that such a form 
can be extended to
any smooth model of $\Gamma\setminus \HH_2$.
Let $\Gamma$ be an arbitrary subgroup of $\Sp_2(\RR)$, which contains a 
principal congruence subgroup $\Gamma_1(N)\subset \Sp_2(\ZZ)$ of some level
$N$. We then have the following

\begin{proposition}[Hilfsatz 3.2.1. in \cite{F}]
An element
$$
\omega_F=F(Z)dZ \in H^0((\Gamma\setminus\HH_2)^o,\, 
\Omega_3((\Gamma\setminus\HH_2)^o))
$$ 
can be extended to a canonical 
differential form on a smooth compactification
$\overline{\Gamma\setminus\HH_2}$
if and only if the differential form $\omega_F$ is square-integrable.
\end{proposition}

It is well-known that a $\Gamma$-invariant differential form
$\omega_F=F(Z) dZ$ is square-integrable if and only if $F$ is a cusp
form of weight $3$.
Thus we have
the following identity for the geometric genus of the variety
$$
h^{3,0}(\overline{\Gamma\setminus\HH_2})
=\hbox{dim}_\CC\, S_3(\Gamma).
$$ 
In particular, when $t$ is prime, we 
have $\Gamma_t^*=\Gamma_t^+$ and the space $S_3(\Gamma_t^*)$ is just the 
space of antisymmetric cusp forms of weight $3$.

\begin{theorem}
The moduli space ${\mathcal K_p}=\Gamma_p^* \setminus \HH_2$ of Kummer surfaces associated to $(1,p)$-polarised abelian surfaces has positive geometric genus for all prime $p=N(\bold a)$ from Theorem \ref{th:wt3}.
In particular, it is positive for  $t=167$, $173$, $223$, $227$, $251$, $257$, $269$, $271$, $283$, $293$.  Moreover, we have 
$$
h^{3,0}(\Gamma_t^*,\CC) \geq 2, \quad t=227, 257, 269, 283,\quad
{\rm and }\quad 
h^{3,0}(\Gamma_{293}^*,\CC) \geq 4.
$$
For all square-free $t=N(\bold a)$ from Theorem \ref{th:wt3}, the moduli space ${\mathcal A}_t^+=\Gamma_t^+ \setminus \HH_2$  of polarized $K3$ surfaces with
a polarisation of type $\latt{2t}\oplus\, 2E_8(-1)$ has positive geometric genus. The minimal such $t$ equals $122$. (See Tables \ref{tableanti3prime}, \ref{tableanti3(1)}, \ref{tableanti3(2)}.)
\end{theorem}

It is known that $\dim S_3(\Gamma_t^*)=0$ for $t\le 40$ (see \cite{BPY}).
According to the calculation made by Jerry Shurman,
the minimal level $t$ with $\dim S_3(\Gamma_t^*)\neq 0$ is $152$ or 
$167$. For $t=152=8\times19$ the antisymmetric paramodular form of weight 
$3$ and level $152$ starts with the theta block 
$\vartheta^5\vartheta_2^4\vartheta_3^4\vartheta_4^3
\vartheta_5^2\vartheta_6\vartheta_7
\vartheta_8/\eta^{15}$. 
We leave to the readers two questions on this paramodular form.
{\it Does it belong to  $M_3(\Gamma_{152}^*)$? Is it a cusp form?}

We note that antisymmetric paramodular forms of weight $3$ 
conjecturally occur as cohomology classes in $H^5(\Gamma_0(N), \CC)$  
studied by  Ash, Gunnells and McConnell in \cite{AGM},
where $\Gamma_0(N) \subseteq \SL_4(\ZZ)$ is defined by having a last row 
in $(N\ZZ, N\ZZ, N\ZZ, \ZZ)$. 

We hope to get more geometric applications 
of  antisymmetric forms of 
weights $3$ and $4$ (see \S \ref{sec:wt4}) in the near future.

\subsection{Automorphic L-functions}
For the prime polarisation $p$ such that the space 
$S_3(\Gamma_p^{+})$ is one dimensional, we get a new eigenfunction of all Hecke operators. (Compare with the Igusa modular form $\Delta_{35}$.) We conjecture that the first such prime is $167$.  For $t=122$,
the antisymmetric cusp form is an oldform, and comes from a newform 
in $S_3(\Gamma_{61})$.  Hypothetically its $Spin$-$L$-function coincides
with motivic $L$-function of a non-rigid Calabi--Yau threefold.

Antisymmetric paramodular cusp forms of weight $2$ are  also very 
interesting. For a prime polarisation, such a form might exist only for 
$p=587$ (see \cite{PY}). At the moment, only three examples are known (see \cite{GPY2}) for $t=587$, $713$, $893$. The first one supports 
the Paramodular Conjecture of Brumer and Kramer.
 Unfortunately there is no antisymmetric reflective modular form of 
singular weight for a lattice of signature $(2,6)$ which splits two 
integral (renormalised) hyperbolic planes at present (see 
\cite{Sch06,Sch17}). In addition, the leading Fourier--Jacobi coefficients 
of such antisymmetric paramodular forms are theta blocks of 
weight $2$ with vanishing order $>1$ in $q$. But so far, no such infinite 
family of theta blocks has been found (see \cite{GSZ}).  Therefore, we cannot construct an infinite series of antisymmetric paramodular forms of weight $2$ using the approach of this paper.

\subsection{Hyperbolization of affine Lie algebras} 
One can put the following question: to find all Lorentzian Kac--Moody 
algebras whose Kac--Weyl--Borcherds denominator functions written at a one-dimensional cusp coincides with the Kac--Weyl denominator function of an 
affine Lie algebra. For such algebras,  one can study the
Lorentzian--Kac Moody Lie algebra as a module over the corresponding 
affine Lie algebra.  

The Kac--Weyl denominator function of an 
affine Lie algebra $\hat{\mathfrak{g}}(R)$ 
for a positive definite $2$-root 
system $R$ of rank $n$ is the following theta block
\begin{equation}\label{KWD}
\psi_R(\tau,{\mathfrak z})=\eta(\tau)^n\prod_{r \in R>0}
   \frac{\vartheta(\tau,(r,{\mathfrak z}))}{\eta(\tau)}\,,
\end{equation}
where the product is taken over all positive roots of the system~$R$
and $\mathfrak z\in R\otimes \CC$. 
This is a Jacobi form for the lattice $R^\vee(h)$, where $h=|R|/n$
is the Coxeter number of $R$ (see \cite[\S 2]{G18}).
In particular, the odd Jacobi theta-series $\vartheta(\tau,z)$
(see \eqref{theta}) is the Kac--Weyl denominator function
of $\hat{\mathfrak{g}}(A_1)$.

The possible list of the Lorentzian Kac--Moody 
algebras which are hyperbolizations of the affine Lie algebras is rather 
short. They are 
the affine algebras for $A_1$ (see \cite{GNII}),
$2A_1$, $4A_1$, $A_2$, $3A_2$ and $23$ root systems of the Niemeier 
lattices of rank $24$ (see \cite{GN17}, \cite{G18}, \cite{GSZ}),
$A_4$ (see \cite{GW} and \cite{GW18}). 
The function $\Phi_3^{\Sch}$ gives the case of $A_6$.
We note that the Kac--Weyl--Borcherds denominator function of the 
Lorentzian Kac--Moody algebra for the cases 
$R=A_1$, $2A_1$, $4A_1$, $A_2$, $3A_2$ and $A_4$ is the Gritsenko lifting of the corresponding Kac--Weyl denominator function of 
$\hat{\mathfrak g}(R)$. 
At the end of the paper we consider another function of Scheithauer
which gives a hyperbolization of the affine Lie algebra 
$\hat{\mathfrak g}(A_4\oplus A_4)$. This interpretation gives 
antisymmetric paramodular forms of weight $4$. We are planning to apply 
them to algebraic geometry soon.

\section{Antisymmetric paramodular forms of weight 4}\label{sec:wt4}
According to Scheithauer's work (see \cite[Theorem 10.3]{Sch06}), there is a reflective Borcherds product of singular weight $4$ with respect to  the lattice
\begin{equation*}
U\oplus U(5)\oplus \text{Maass lattice},
\end{equation*}
whose genus is $\II_{10,2}(5^{+6})$. We can check  that
\begin{equation*}
U\oplus U(5)\oplus \text{Maass lattice} \cong 2U\oplus A_4^\vee(5)\oplus A_4^\vee(5).
\end{equation*}
The explicit description and more properties of the lattice $A_4^\vee(5)$ can be found in our last preprint \cite{GW18}.

We can reconstruct the Borcherds product on 
$U\oplus U(5)\oplus (\text{Maass lattice})$ at the 1-dimensional cusp related to $2U\oplus A_4^\vee(5)\oplus A_4^\vee(5)$. 
By Proposition \ref{proplift},  we have
\begin{equation}
\begin{split}
\Psi_{2A_4^\vee(5)}(\tau,\mathfrak{z})=&\Psi_{\Gamma_0(5),\eta^{-4}(\tau)\eta^{-4}(5\tau),0}\\
=&q^{-1}+\sum_{\substack{r\in A_4\oplus A_4\\ (r,r)=2}}e^{2\pi i(r,\mathfrak{z})}+8+O(q) \in J_{0,2A_4^\vee(5),1}^{!,\Orth(2A_4)}.
\end{split}
\end{equation}
Therefore, the function $\Borch(\Psi_{2A_4^\vee(5)})$ is a reflective modular form of weight 4 with respect to $\Orth^+(2U\oplus 2A_4^\vee(-5))$ with divisor
\begin{equation}
\div (\Borch(\Psi_{2A_4^\vee(5)})) = \sum_{\substack{r\in 2U\oplus 2A_4^\vee(5) \\ (r,r)_2=-2}} \cD_r + \sum_{\substack{s\in 2U\oplus \frac{1}{5} 2A_4(-1)\\ (s,s)_2=-\frac{2}{5}}} \cD_s.
\end{equation}
Moreover, the character of $\Borch(\Psi_{2A_4^\vee(5)})$ for the group $\widetilde{\Orth}^+(2U\oplus 2A_4^\vee(-5))$ is $\det$. The first Fourier--Jacobi coefficient of $\Borch(\Psi_{2A_4^\vee(5)})$ is equal to the Kac--Weyl denominator function of the affine Lie algebra $\hat{\mathfrak g}(2A_4)$. Similar to \S \ref{sec:wt3}, we obtain the following theorem.

\begin{theorem}\label{th:wt4}
Given $\mathbf{a}=(a_1,a_2,a_3,a_4)\in\ZZ^4$, $\mathbf{b}=(b_1,b_2,b_3,b_4)\in \ZZ^4$. Let $n_0(\mathbf{a},\mathbf{b})$ be the number of $0$ in the following $20$ integers
\begin{equation}\label{list4}
\begin{split}
&a_1, a_2, a_3, a_4, a_1+a_2, a_2+a_3, a_3+a_4,a_1+a_2+a_3,  \\
&a_2+a_3+a_4,a_1+a_2+a_3+a_4,\\
&b_1, b_2, b_3, b_4, b_1+b_2, b_2+b_3, b_3+b_4,b_1+b_2+b_3,  \\
&b_2+b_3+b_4,b_1+b_2+b_3+b_4.
\end{split}
\end{equation}
Denote by $N(\mathbf{a},\mathbf{b})$ the half of the sum of the squares of the above $20$ integers.
We define a weakly holomorphic Jacobi form in one variable
\begin{equation}
\Psi_{2A_4^\vee(5),\mathbf{a},\mathbf{b}}(\tau,z)=\Psi_{2A_4^\vee(5)}\biggl(\tau, z\sum_{i=1}^4a_iu_i+z\sum_{j=1}^4b_jv_j \biggr),
\end{equation}
where $u_i$ are the fundamental weights of the first copy of $A_4$ and $v_j$ are the fundamental weights of the second copy of $A_4$.
Then $\Borch(\Psi_{2A_4^\vee(5),\mathbf{a},\mathbf{b}})$ is a holomorphic antisymmetric Siegel modular form of weight $4+n_0(\mathbf{a},\mathbf{b})$ with respect to the paramodular group of level $N(\mathbf{a},\mathbf{b})$. Moreover, the first Fourier-Jacobi coefficient of $\Borch(\Psi_{2A_4^\vee(5),\mathbf{a},\mathbf{b}})$ is equal to 
\begin{equation*}
\eta^{3n_0(\mathbf{a},\mathbf{b})-12}\prod_{c}
\vartheta(\tau, cz),
\end{equation*}
where the product runs over all non-zero integers in the list \eqref{list4}.
\end{theorem}

We remark that all antisymmetric paramodular cusp forms of weights large than $3$ constructed in \cite{GPY2} can be reconstructed by our method. We list all of them and many new examples in Table \ref{tableanti4}.


\begin{table}[ht]
\caption{Antisymmetric paramodular cusp forms of weights large than 3}
\label{tableanti4}
\renewcommand\arraystretch{1.5}
\noindent\[
\begin{array}{|c|c|c|c|}
\hline 
\text{weight} & N(\mathbf{a},\mathbf{b}) & \mathbf{a}, \mathbf{b} & \text{Theta block} \\ 
\hline 
4 & 62 & (1,1,1,1), (2,1,1,1)& \vartheta^7\vartheta_2^6\vartheta_3^4\vartheta_4^2\vartheta_5/\eta^{12} \\
\hline 
5 & 38 & (1,1,1,1), (-1,1,1,1)& \vartheta^9\vartheta_2^6\vartheta_3^3\vartheta_4/\eta^{9}\\
\hline 
5 & 42 & (1,1,1,1), (0,1,1,1)& \vartheta^8\vartheta_2^6\vartheta_3^4\vartheta_4/\eta^{9} \\
\hline 
5 & 53 & (1,1,1,1), (0,1,1,2)& \vartheta^7\vartheta_2^6\vartheta_3^3\vartheta_4^3/\eta^{9}\\
\hline 
5 & 65 & (2,1,1,1), (0,1,1,2)&\vartheta^6\vartheta_2^6\vartheta_3^3\vartheta_4^3\vartheta_5/\eta^{9} \\
\hline 
6 & 26 & (-1,1,1,1), (-1,1,1,1)& \vartheta^{10}\vartheta_2^6\vartheta_3^2/\eta^6 \\
\hline 
7 & 23 & (-1,1,1,1),(0,1,1,0) &\vartheta^{9}\vartheta_2^7\vartheta_3/\eta^3\\
\hline 
8 & 14 & (1,-1,1,1), (1,-1,1,1)&\vartheta^{12}\vartheta_2^4 \\
\hline 
8 & 17 & (0,1,1,0), (1,-1,1,1)& \vartheta^{10}\vartheta_2^6\\
\hline 
9 & 15 & (0,0,1,1), (1,-1,1,1)& \eta^3\vartheta^{10}\vartheta_2^5 \\
\hline 
\end{array} 
\]
\end{table}

\begin{remark}
We can also consider the pull-back to a lattice of signature $(2,4)$. We define a weakly holomorphic Jacobi form for a lattice of rank $2$
\begin{equation*}
\Psi_{2A_4^\vee(5),\mathbf{a}+\mathbf{b}}(\tau,z_1,z_2)=\Psi_{2A_4^\vee(5)}\biggr(\tau, z_1\sum_{i=1}^4a_iu_i+z_2\sum_{j=1}^4b_jv_j \biggl).
\end{equation*}
Assume that $n_0(\mathbf{a},\mathbf{b})=0$. Denote by $N_0(\mathbf{a})$ the half of the sum of the squares of the first $10$ integers about $\mathbf{a}$ and by $N_0(\mathbf{b})$ the half of the sum of the squares of the last $10$ integers about $\mathbf{b}$. Then the Borcherds product $\Borch(\Psi_{2A_4^\vee(5),\mathbf{a}+\mathbf{b}})$ will give an antisymmetric holomorphic modular form of canonical weight $4$ for the stable orthogonal group of the lattice $2U\oplus \latt{-2N_0(\mathbf{a})}\oplus \latt{-2N_0(\mathbf{b})}$. We hope that this type of modular forms would have applications in Hermitian modular forms and in corresponding modular varieties. It will be interesting to seek a similar test to check the cuspidality of the constructed modular forms as in Proposition \ref{prop:test}. 
\end{remark}
\smallskip

\noindent
\textbf{Acknowledgements.} 
Both  authors are supported by the Laboratory of Mirror Symmetry NRU HSE 
(RF government grant, ag. N 14.641.31.0001).
The second author is also supported by the Labex CEMPI (ANR-11-
LABX-0007-01) in the University of Lille.

\bibliographystyle{amsplain}

\begin{thebibliography}{10}

\bibitem{AGM} 
A. Ash, P. E. Gunnells, M. McConnell,
\textit{Cohomology of congruence subgroups of $SL(4;\ZZ)$. III.} 
Math. Comp. \textbf{79} (2010), 1811--1831.


\bibitem{Bo95} R. E. Borcherds, {\it Automorphic forms on 
$\Orth_{s+2,2}(\RR)$ and infinite products}, Invent. Math. {\bf 120:1} (1995), 161--213.

\bibitem{Bo98} R. E. Borcherds, \textit{Automorphic forms with singularities on Grassmannians.} Invent. Math., \textbf{123} (1998), no. 3, 491--562.

\bibitem{Bou60} N. Bourbaki, {\it Groupes et alg\`{e}bres de Lie.} Chapter 4,5 et 6.

\bibitem{BPY} J. Breeding II, C. Poor, D. S. Yuen,
\textit{Computations of spaces of paramodular
forms of general level.} J. Korean Math. Soc. \textbf{53} (2016),
645--689.

\bibitem{BK}  A. Brumer, K. Kramer, \textit{Paramodular abelian varieties of odd conductor.} Trans. Amer. Math. Soc. \textbf{366} (2014), 2463--2516.

\bibitem{CG} F. Cl\'ery,   V. Gritsenko,
\textit{Modular forms of orthogonal type and Jacobi theta-series.}
Abh. Math. Semin. Univ. Hambg \textbf{83} (2013), 187--217.

\bibitem{EZ} M. Eichler, D. Zagier, \textit{The Theory of Jacobi
Forms.} Progress in Mathematics \textbf{55}. Birkh\"auser, Boston,
Mass., 1985.

\bibitem{F} E. Freitag, \textit{Siegelsche Modulfunktionen.} Grundlehren der mathematischen Wissenschaften 254, Springer-Verlag (1983).

\bibitem{G94} V. Gritsenko, \textit{Irrationality of the moduli spaces of polarized abelian surfaces.} Int. Math. Res. Not. IMRN \textbf{6} (1994), 235--243.

\bibitem{G18} V. Gritsenko, \textit{Reflective modular forms and their applications.} Russian Math. Surveys \textbf{73:5} (2018), 
797--864.

\bibitem{GH98} V. Gritsenko, K. Hulek, \textit{Minimal Siegel modular threefolds.} Math. Proc. Cambridge Philos. Soc. \textbf{123} (1998) 461--485.

\bibitem{GH14}
V. Gritsenko, K. Hulek,
{\it Uniruledness of orthogonal modular varieties.} J. of
Algebraic Geom. \textbf{23} (2014), 711--725.

\bibitem{GHS1}
V. Gritsenko, K. Hulek, G. K. Sankaran,
{\it Abelianisation of orthogonal groups
and the fundamental group of modular varieties}.
J. Algebra {\bf 322:2} (2009), 463--478.

\bibitem{GHS2}
V. Gritsenko, K. Hulek, G. K. Sankaran, 
{\it The Kodaira dimension of the
moduli of K3 surfaces.} Invent. Math. {\bf 169:3} (2007), 519--567.

\bibitem{GN96}
V. Gritsenko, V. V. Nikulin, 
{\it Igusa modular forms and ‘the simplest’
Lorentzian Kac–Moody algebras}, Sb. Math. {\bf 187:11} (1996), 
1601–-1641.

\bibitem{GNII} V. Gritsenko, V. Nikulin, 
\textit{Automorphic forms and Lorentzian Kac--Moody algebras. Part II.}  Internat. J. Math. \textbf{9} (1998), 201--275.

\bibitem{GN17} V. Gritsenko, V. V. Nikulin,
{\it Lorentzian Kac–-Moody algebras with Weyl
groups of $2$-reflections.} 
Proc. Lond. Math. Soc. {\bf 116:3} (2018), 485--533. 

\bibitem{GPY2} V. Gritsenko, C. Poor, D. S. Yuen, \textit{Antisymmetric paramodular forms of weights $2$ and $3$.} Int. Math. Res. Not. IMRN, https://doi.org/10.1093/imrn/rnz011.


\bibitem{GSZ} V. Gritsenko, N. P. Skoruppa, D. Zagier, 
\textit{Theta blocks}, preprint 2018, 56 pp. https://math.univ-lille1.fr/d7/sites/
default/files/THETA$\%$20BLOCKS22.09.18$\underline{\ \ }$1.pdf

\bibitem{GW} V. Gritsenko, H. Wang,
\textit{Conjecture on theta-blocks of order $1$.}
 Russian Math. Surveys \textbf{72:5} (2017), 968--970.
 
\bibitem{GW18} V. Gritsenko, H. Wang,
\textit{Theta block conjecture for paramodular forms of weight $2$.} Preprint 2018, 15 pp, arXiv:1812.08698.


\bibitem{GP}
M. Gross, S. Popescu,
\textit{Calabi--Yau three-folds and moduli of abelian surfaces II.}
Trans. Amer. Math. Soc. \textbf{363} (2011), 3573--3599.


\bibitem{Nik80} V. Nikulin, 
\textit{Integer symmetric bilinear forms and some of their geometric applications.} Math. USSR Izv. \textbf{14} (1980), 103--167. 

\bibitem{PSY} C. Poor, J. Shurman, D. S. Yuen, 
\textit{Siegel paramodular forms of weight $2$ and squarefree level.} Int. 
J. Number Theory \textbf{13} (2017), 2627--2652.

\bibitem{PY} C. Poor, D. S. Yuen,
{\it Paramodular Cusp Forms},
Math. Comp. {\bf 84} (no. 293) (2015), 1401--1438.


\bibitem{Sch06} N. R. Scheithauer, 
\textit{On the classification of automorphic products and generalized Kac-Moody algebras.} Invent. Math. \textbf{164} (2006), 641--678.

\bibitem {Sch09} N. R. Scheithauer, 
\textit{The Weil representation of $\SL_2(\ZZ)$ and some applications.} Int. Math. Res. Not. IMRN (2009) no.8, 1488--1545.

\bibitem{Sch15} N. R. Scheithauer, 
\textit{Some constructions of modular forms for the Weil representation of $\SL(2,\ZZ)$.} Nagoya Math. J. \textbf{220} (2015), 1--43.

\bibitem{Sch17} N. R. Scheithauer,
\textit{Automorphic products of singular weight.}  Compos. Math. \textbf{153} (2017), 1855--1892. 

\end{thebibliography}

\end{document}